\documentclass[a4paper,11pt]{amsart}
\usepackage[utf8]{inputenc}
\usepackage{amsaddr}
\usepackage{calrsfs}
\usepackage{graphicx,color}
\usepackage{subfig}
\numberwithin{equation}{section}

\theoremstyle{definition}

\theoremstyle{plain}
\newtheorem{thm}{Theorem}[section]
\newtheorem{pro}{Proposition}[section]

\newtheorem{lmm}{Lemma}[section]
\theoremstyle{definition}
\newtheorem{rem}{Remark}[section]

\newcommand{\R}{\mathbb{R}}
\newcommand{\E}{\mathbb{E}}
\renewcommand{\P}{\mathbb{P}}

\newcommand{\F}{\mathcal{F}}

\newcommand{\K}{\mathfrak{K}}

\newcommand{\B}{\mathrm{B}}

\newcommand{\la}{\langle}
\newcommand{\ra}{\rangle}

\newcommand{\1}{\mathbf{1}}
\renewcommand{\d}{\mathrm{d}}

\begin{document}
\title[Prediction law of fractional Brownian motion]
{Prediction law of fractional Brownian motion}

\date{\today}

\author[Sottinen]{Tommi Sottinen}
\address{Department of Mathematics and Statistics, University of Vaasa, P.O. Box 700, FIN-65101 Vaasa, FINLAND}
\email{tommi.sottinen@iki.fi}

\author[Viitasaari]{Lauri Viitasaari}
\address{Department of Mathematics and System Analysis, Aalto University School of Science, Helsinki, P.O. Box 11100, FIN-00076 Aalto,  FINLAND} 
\email{lauri.viitasaari@aalto.fi}

\begin{abstract}
We calculate the regular conditional future law of the fractional Brownian motion with index $H\in(0,1)$ conditioned on its past.  We show that the conditional law is continuous with respect to the conditioning path.  
We investigate the path properties of the conditional process and the asymptotic behavior of the conditional covariance.
\end{abstract}

\thanks{T. Sottinen was partially funded by the Finnish Cultural Foundation (National Foundations' Professor Pool).
\\ \indent
L.Viitasaari was partially funded by the Emil Aaltonen Foundation.}



\maketitle

%%%%%%%%%%%%%%%%%%%%%%%%%%%%%%%%%%%%%%%%%%%%%%%%%%%%%%%%%%%%%%%%%%%%%%%%%%%%%%%
%%%%%%%%%%%%%%%%%%%%%%%%%%%%%%%%%%%%%%%%%%%%%%%%%%%%%%%%%%%%%%%%%%%%%%%%%%%%%%%
\section{Introduction}

Let $B^H=(B^H_t)_{t\in\R_+}$ be the fractional Brownian motion with Hurst index $H\in(0,1)$. Let $u\in\R_+$, and let $\F_u$ be the $\sigma$-field generated by the fractional Brownian motion on the interval $[0,u]$.  We study the prediction of $(B_t)_{t\ge u}$ given the information $\F_u$.  In other words, we study the conditional regular law of $\hat B^{H}(u) = B^H|\F_u$. It is well-known that such regular conditional laws for Gaussian processes exists and they are Gaussian with random conditional mean and deterministic conditional variance, see, e.g., Bogachev \cite[Section 3.10]{Bogachev-1998} or Janson \cite[Chapter 9]{Janson-1997}. Recently, LaGatta \cite{LaGatta-2013} introduced the notion of continuous disintegration. In our case it reads as follows: Let $T>0$ be arbitrary and let $\P_T$ be the law of the fractional Brownian motion on $[0,T]$. The regular conditional law $\P^y_T = \P_T[\,\cdot\,| B_v^H = y(v), v\le u]$ is continuous with respect to $y$ if $y_n\to y$ (in sup-norm) implies $\P_T^{y_n}\to \P_T^y$ (weakly).  
%\textcolor{red}{seuraavan lauseen lisasin kohtaan 3} 
In practice however, the conditioning trajectory is usually observed only at discrete points, and the continuity of regular conditional law in the sense of LaGatta has practical importance. 
We calculate the regular conditional law of fractional Brownian motion explicitly and show that it is continuous with respect to the conditioning trajectory.

%\textcolor{red}{kirjoitin seuraavan uusiksi kokonaan, muokkaa tarvittaessa. Taman tarkoitus on hoitaa refereiden muut kohdat paitsi tuo lagatta (ylla) ja kuvat}
Prediction of fractional Brownian motion has long history, see Molchan \cite{Molchan-2003}.  For example, Gripenberg and Norros \cite{Gripenberg-Norros-1996} provided the conditional mean of the fractional Brownian motion with parameter $H>\frac{1}{2}$ based on observations extending to the infinite past, Norros et al. \cite{Norros-Valkeila-Virtamo-1999} provided the conditional mean for the whole range $H\in(0,1)$ based on the observations on a compact interval, Anh and Inoue \cite{Anh-Inoue-2004} provided the conditional mean for $H<\frac12$, Duncan \cite{Duncan-2006} studied conditional mean for some processes related to integrals of fractional Brownian motion, and Fink et al. \cite{Fink-Kluppelberg-Zahle-2013} derived conditional characteristic function and the conditional variance. However, it seems that the regular conditional law itself has not been studied.

%%%%%%%%%%%%%%%%%%%%%%%%%%%%%%%%%%%%%%%%%%%%%%%%%%%%%%%%%%%%%%%%%%%%%%%%%%%%%%%
\section{Preliminaries}

We recall some facts for fractional Brownian motion and fractional calculus. As general references for fractional Brownian motion we refer to Biagini et al. \cite{Biagini-Hu-Oksendal-Zhang-2008} and  Mishura \cite{Mishura-2008}.  The standard reference for fractional calculus is Samko et al.  \cite{Samko-Kilbas-Marichev-1993}.  
%For the specific results presented here we refer to Pipiras and Taqqu \cite{Pipiras-Taqqu-2001}.

The fractional Brownian motion $B^H=(B^H_t)_{t\in\R_+}$ with Hurst index $H\in(0,1)$ is the centered Gaussian process with covariance
$$
r_H(t,s) = \frac{1}{2}\left[t^{2H}+s^{2H}- |t-s|^{2H}\right].
$$
The case $H>\frac{1}{2}$ corresponds to the long-range dependent case, or positively correlated increments.  The case $H<\frac{1}{2}$ corresponds to the short-range dependent case, or negatively correlated increments.  For $H=\frac{1}{2}$, we have the classical Brownian motion.

Let
$$
I_{t-}^\alpha[f] (s) :=
\frac{1}{\Gamma(\alpha)}\int_s^t f(z)(z-s)^{\alpha-1}\, \d z, \quad s\in(0,t),
$$
be the right-sided fractional integral of order $\alpha\in(0,1)$. The inverse of $I_{t-}^\alpha$ is the right-sided fractional derivative
$$
I_{t-}^{-\alpha}[f](s) :=
-\frac{1}{\Gamma(1-\alpha)}\frac{\d}{\d s}\int_s^t f(z)(z-s)^{-\alpha}\, \d z.
$$

For $u\in\R_+$, define
$$
\K_{H,u}[f](t) := \sigma_H t^{H-\frac{1}{2}}I_{u-}^{H-\frac{1}{2}}\left[(\cdot)^{H-\frac{1}{2}} f\right](t),
$$
where
$$
\sigma_H
= \sqrt{\frac{\pi(H-\frac{1}{2})2H}{\Gamma(2-2H)\sin(\pi(H-\frac{1}{2}))}}.
$$
Let $\K_{H,u}^{-1}$ be the inverse of $\K_{H,u}$, i.e.,
$$
\K_{H,u}^{-1}[f](t) = \frac{1}{\sigma_H} t^{\frac{1}{2}-H} I_{u-}^{\frac{1}{2} -H}\left[(\cdot)^{H-\frac{1}{2}} f\right](t).
$$
Indicator functions $\1_{[0,t)}$ belong to the domains of $\K_{H,t}$ and $\K_{H,t}^{-1}$ for all $t>0$ and $H\in(0,1)$. So, we can define
\begin{eqnarray*}
k_H(t,s) &:=& \K_{H,t}\left[\1_{[0,t)}\right](s), \\
k^{-1}_H(t,s) &:=& \K_{H,t}^{-1}\left[\1_{[0,t)}\right](s).
\end{eqnarray*}
Then
\begin{equation}
\label{eq:kernel}
k_H(t,s) = d_H\left[\left(\frac{t}{s}\right)^{H-\frac{1}{2}}(t-s)^{H-\frac12}
-(H-\frac{1}{2})s^{\frac{1}{2}-H}\int_s^t z^{H-\frac{3}{2}}(z-s)^{H-\frac{1}{2}}\, \d z\right],  
\end{equation}
where
$$
d_H = \sqrt{\frac{2H\Gamma(\frac{3}{2}-H)}{\Gamma(H+\frac{1}{2})\Gamma(2-2H)}}.
$$
(A similar formula can be found for $k_H^{-1}$ also, but we have no need for it here.)

\begin{lmm}[Volterra Correspondence]\label{lmm:vc}
Let $B^H$ be a fractional Brownian motion. Then the process
$$%\begin{equation}\label{eq:bm-from-fbm}
W_t = \int_0^t k_H^{-1}(t,s)\, \d B^H_s
$$%\end{equation}
is a Brownian motion.  Moreover, the fractional Brownian motion can be recovered from it by 
$$%\begin{equation}\label{eq:fbm-from-bm}
B^H_t = \int_0^t k_H(t,s)\, \d W_s.
$$%\end{equation}
\end{lmm}

The Volterra correspondence of Lemma \ref{lmm:vc} above extends to a transfer principle of Lemma \ref{lmm:tp} below. We note that in Lemma \ref{lmm:vc} can be taken as the definition of an abstract Wiener integral with respect to the fractional Brownian motion.  We refer to Pipiras and Taqqu \cite{Pipiras-Taqqu-2001} for details on Wiener integration with respect to fractional Brownian motions, and for \cite{Sottinen-Yazigi-2014} for a more general discussion on abstract Wiener integration.

\begin{lmm}[Transfer Principle]\label{lmm:tp}
Let $B^H$ and $W$ be as in Lemma \ref{lmm:vc}. Then, for all $u\in\R_+$,
$$%\begin{equation}\label{eq:transfer-to-fbm}
\int_0^u f(t)\, \d W_t
= \int_0^u \K_{H,u}^{-1}[f](t)\, \d B^H_t
$$%\end{equation} 
for all $f\in L^2([0,u])$, and
$$%\begin{equation}\label{eq:transfer-to-bm}
\int_0^u f(t)\, \d B^H_t
= \int_0^u \K_{H,u}[f](t)\, \d W_t,
$$%\end{equation} 
for all $f\in \K_{H,u}^{-1} L^2([0,u])$.
\end{lmm}

%%%%%%%%%%%%%%%%%%%%%%%%%%%%%%%%%%%%%%%%%%%%%%%%%%%%%%%%%%%%%%%%%%%%%%%%%%%%%%%
\section{Regular Conditional Law}

\begin{thm}[Prediction Law]\label{thm:cond_law}
The conditional process $\hat B^{H}(u)= (\hat B^H_t(u))_{t\ge u}$ is Gaussian with $\F_u$-measurable mean function
\begin{equation}\label{eq:cond_mean}
\hat m^H_t(u) = B^H_u - \int_0^u \Psi_H(t,s|u) \, \d B^H_s,
\end{equation}
where
$$
\Psi_H(t,s|u) =
- \frac{\sin(\pi(H-\frac{1}{2}))}{\pi} s^{\frac{1}{2}-H}(u-s)^{\frac{1}{2}-H}
\int_u^t \frac{z^{H-\frac{1}{2}}(z-u)^{H-\frac{1}{2}}}{z-s}\, \d z,
$$
and deterministic covariance function
\begin{equation}
\label{eq:cond_cov}
\hat r_H(t,s|u) 
%= \int_u^{t\wedge s}k_H(t,v)k_H(s,v)\d v
= r_H(t,s) - \int_0^u k_H(t,v)k_H(s,v)\d v.
\end{equation}
Moreover, the regular conditional law is continuous with respect to the conditioning trajectory $(B^H_v)_{v\le u}$.
\end{thm}
\begin{proof}
Let $B^H$ and $W$ be as in Lemma \ref{lmm:vc}. Let $t\ge u$. Then
\begin{eqnarray*}
\hat m_t(u) &=& \E\left[B^H_t\, \big|\, \F_u\right] \\
&=&
\E\left[\int_0^t k_H(t,s)\, \d W_s \, \bigg|\, \F_u^W\right] \\
&=& 
\int_0^u k_H(t,s)\, \d W_s \\
&=&
\int_0^u k_H(u,s)\, \d W_s - \int_0^u \left[k_H(u,s)-k_H(t,s)\right] \d W_s \\
&=&
B^H_u - \int_0^u \left[k_H(u,s)-k_H(t,s)\right] \d W_s.
\end{eqnarray*}
It remains to show that the function $s\mapsto k_H(u,s)-k_H(t,s)$, $s\in [0,u]$, belongs to $\K_{H,u}^{-1} L^2([0,u])$ and then to apply Lemma \ref{lmm:tp} and calculate the transfered kernel for the equation
$$
\int_0^u \left[k_H(u,s)-k_H(t,s)\right] \d W_s
=
\int_0^u \K_{H,u}^{-1}\left[k_H(u,\cdot)-k_H(t,\cdot)\right](s)\, \d B^H_s.
$$
This was done in Pipiras and Taqqu \cite[Theorem 7.1]{Pipiras-Taqqu-2001}. 

Let us then calculate the conditional covariance.  Let $W$ be as before. Then
\begin{eqnarray*}
\hat r_H(t,s|u) &=&
\E\left[\left(B^H_t-\hat m_t(u)\right)\left(B^H_s-\hat m_s(u)\right)\,\big|\, \F_u\right] \\
&=&
\E\left[
\left(\int_0^t k_H(t,v)\, \d W_v - \int_0^u k_H(t,v)\, \d W_v\right)\times\right. \\
& & \left.
\left(\int_0^s k_H(s,w)\, \d W_w - \int_0^u k_H(s,w)\, \d W_w\right)
\,\bigg|\, \F_u^W\right] \\
&=&
\E\left[
\int_u^t k_H(t,v)\, \d W_v
\int_u^s k_H(s,w)\, \d W_w 
\,\bigg|\, \F_u^W\right] \\
&=&
\int_u^{t\wedge s} k_H(t,v) k_H(s,v)\, \d v \\
&=&
r_H(t,s) - \int_0^u k_H(t,v) k_H(s,v)\, \d v,
\end{eqnarray*}
where the kernel $k_H(t,s)$ is given by \eqref{eq:kernel}.
% \cite[Theorem 5.2]{Norros-Valkeila-Virtamo-1999}. 

Finally, to invoke \cite[Theorem 2.4]{LaGatta-2013}, we must show that
$$
\sup_{v\in[0,u]} \frac{\sup_{t\in [0,T]}|r_H(v,t)|}{\sup_{w\in [0,u]}|r_H(v,w)|}
< \infty
$$
for all $T>0$. Since $r_H$ is continuous,
$$
\frac{\sup_{t\in [0,T]}|r_H(v,t)|}{\sup_{w\in [0,u]}|r_H(v,w)|}
=
\frac{|r_H(v,t^*_{T,v})|}{|r_H(v,w^*_{u,v})|} \\
\le
\frac{|r_H(v,t^*_{T,v})|}{|r_H(v,v)|}. 
$$
The ratio above is obviously bounded for all $v\in[\varepsilon,u]$ for any $\varepsilon>0$. As for $v\to 0$, 
$$
\limsup_{v\to 0} \frac{|r_H(v,t^*_{T,v})|}{|r_H(v,v)|} \le 1,
$$
since $r_H(v,t_{T,v}^*) = \sup_{t\in[0,T]} r_H(v,t) \ge r_H(v,v)$.
\end{proof}

\begin{rem}[Brownian Motion]\label{rem:bm}
For $H=\frac{1}{2}$, we have $k_{\frac{1}{2}}(t,s) = \1_{[0,t)}(s)$ and $\K_{\frac{1}{2},u}$ is the identity operator.  Consequently, we recover from the proof of Theorem \ref{thm:cond_law} that
\begin{eqnarray*}
\hat m^{\frac{1}{2}}_t(u) &=& W_u, \\
\hat r_{\frac{1}{2}}(t,s|u) &=& t \wedge s - u.
\end{eqnarray*}
\end{rem}

\begin{rem}[Prediction Martingale]
The formula \eqref{eq:cond_mean} for the conditional expectation $\hat m^H_t(u)$ is rather complicated.  Let us note, however, that for each fixed prediction horizon $t>0$, the process $\hat m^H_t(\cdot)$ is a Gaussian martingale on $[0,t]$ with bracket
$$
\d\la \hat m_t(\cdot) \ra_u = k_H(t,u)^2\, \d u.
$$
\end{rem}

%%%%%%%%%%%%%%%%%%%%%%%%%%%%%%%%%%%%%%%%%%%%%%%%%%%%%%%%%%%%%%%%%%%%%%%%%%%%%%%
%% Covariance Stuff                                                          %%
%%%%%%%%%%%%%%%%%%%%%%%%%%%%%%%%%%%%%%%%%%%%%%%%%%%%%%%%%%%%%%%%%%%%%%%%%%%%%%%

Next we investigate the conditional covariance $\hat r_H(t,s\,|u)$ for fixed $s\le t$ as a function of $u\in(0,s)$. The proofs are rather technical and lengthy. For this reason they are postponed into Section \ref{sec:proofs}.

\begin{pro}[Conditional Covariance]\label{pro:func_decreasing}
$\hat r_H(t,s\,|\cdot)$ is infinitely differentiable and strictly decreasing on $(0,s)$ for any $H\in(0,1)$. For $H\in[\frac{1}{2},1)$ it is also convex.
\end{pro}

\begin{rem}[Short-Range Dependent Conditional Covariance]
For $H\in(0,\frac{1}{2})$, $\hat r_H(t,s|\cdot)$ is neither convex nor concave. Indeed, it can be shown that for $H\in(0,\frac{1}{2})$, the kernel $k_H$ is positive and
$$
\lim_{s\to 0+} k_H(t,s) = \lim_{s\to t-} k_H(t,s) = \infty.
$$ 
Therefore
$$
\frac{\partial}{\partial u}\hat r_H(t,s|u) 
=
-k_H(t,u)k(s,u)
$$
is neither increasing nor decreasing in $u$.
\end{rem}

\begin{pro}[No-Information Asymptotics]\label{pro:0_asym}
Let $t\geq s$ be fixed.
\begin{enumerate}
\item For $H<\frac{1}{2}$ we have, as $u\rightarrow 0$,
\begin{equation*}
%\label{eq:0_asym_small_H}
\hat r_H(t,s|u) = r_H(t,s) - C_H u^{2H} + o\left(u^{2H}\right),
\end{equation*}
where 
$$
C_H = \frac{d_H^2}{2H} \left(H-\frac{1}{2}\right)^2 \left(\int_1^\infty w^{H-\frac{3}{2}}(w-1)^{H-\frac{1}{2}}\d w\right)^2.
$$
\item For $H>\frac{1}{2}$ we have, as $u\rightarrow 0$,
\begin{equation*}
%\label{eq:0_asym_large_H}
\hat r_H(t,s|u) = r_H(t,s) - C_{H,t,s}u^{2-2H} + o\left(u^{2-2H}\right),
\end{equation*}
where
$$
C_{H,t,s} = \frac{d_H^2(ts)^{2H-1}}{8-8H}.
$$
\end{enumerate}
\end{pro}

\begin{rem}
It is interesting to note in Proposition \ref{pro:0_asym} the different asymptotic behavior for the long-range dependent case ($H>\frac{1}{2}$) and the short-range dependent case ($H<\frac{1}{2})$. Indeed, for the long-range dependent case the principal term in the ``remaining covariance'' $r_H(t,s)-\hat r_H(t,s|u)$ is $C_H u^{2H}$, where the constant $C_H$ is independent of $t$ and $s$.  In the short-range dependent case the principal term is $C_{H,t,s} u^{2-2H}$. So, the power reverts from $2H$ to $2-2H$ (and, consequently, the principal term remains convex) and the constant depends on $t$ and $s$.
\end{rem}

\begin{pro}[Full-Information Asymptotics]
\label{pro:s_asym}
Let $H\in (0,1) \setminus \left\{\frac{1}{2}\right\}$. Then
\begin{enumerate}
\item for $t=s$, we have, as $u\rightarrow s$,
\begin{equation*}
%\label{eq:var_asym_s}
\hat r_H(s,s|u) = \frac{d_H^2}{2H} (s-u)^{2H} + o\left((s-u)^{2H}\right),
\end{equation*}
\item for $t>s$ we have, as $u\rightarrow s$,
\begin{equation*}
%\label{eq:cor_asym_s}
\hat r_H(t,s|u) = C_{H,t,s}(s-u)^{H+\frac{1}{2}} + o\left((s-u)^{H+\frac{1}{2}}\right),
\end{equation*}
where
$$
C_{H,t,s} = \frac{d_H^2}{H+\frac{1}{2}}\left[\left(\frac{t}{s}\right)^{H-\frac{1}{2}}(t-s)^{H-\frac{1}{2}} 
+ \left(\frac{1}{2} - H\right)s^{H-\frac{1}{2}}\int_1^{\frac{t}{s}}w^{H-\frac{3}{2}}(w-1)^{H-\frac{1}{2}}\right].
$$
\end{enumerate}
\end{pro}

%%%%%%%%%%%%%%%%%%%%%%%%%%%%%%%%%%%%%%%%%%%%%%%%%%%%%%%%%%%%%%%%%%%%%%%%%%%%%%%
%% Hölder Stuff                                                              %%
%%%%%%%%%%%%%%%%%%%%%%%%%%%%%%%%%%%%%%%%%%%%%%%%%%%%%%%%%%%%%%%%%%%%%%%%%%%%%%%

Finally, we examine the sample path continuity of the conditional process. Recall that a process $X=(X_t)_{t\in\R_+}$ is H\"older continuous of order $\gamma$, if for all $T>0$ there exists an almost surely finite random variable $C_T$ such that
\begin{equation}\label{eq:def-holder}
|X_t-X_s| \le C_T|t-s|^{\gamma}
\end{equation}
for all $t,s\le T$. The H\"older index of the process is the supremum of all $\gamma$ such that \eqref{eq:def-holder} holds.

Next we show that the H\"older index of the conditional process $\hat B^H(u)$ and the conditional mean $\hat m^H(u)$ are the same as that of the fractional Brownian motion $B^H$.  This is very important e.g. for pathwise stochastic analysis.

\begin{pro}[H\"older Continuity]
Let $u>0$ be fixed. Then the conditional process $\hat B^{H}(u)$ and the conditional mean $\hat m^H(u)$ both have H\"older index $H$.
\end{pro}

\begin{proof}
Let us first consider the conditional mean $\hat m^H(u)$. Since
$$
\hat m^H_t(u) = \int_0^u k_H(t,v)\,\d W_v,
$$
we have, by the It\^o isometry,
$$
\E\left[\left(\hat m^H_t(u)- \hat m^H_s(u)\right)^2\right] 
= 
\int_0^u \left[k_H(t,v) - k_H(s,v)\right]^2 \,\d v.
$$
Let $s\le t$. By Lemma \ref{lmm:vc} and the It\^o isometry, we have
$$
|t-s|^{2H} = \int_0^t \left[k_H(t,v)-k_H(s,v)\right]^2 \, \d v.
$$
Thus 
$$
\E\left[\left(\hat m^H_t(u)- \hat m^H_s(u)\right)^2\right] 
\leq |t-s|^{2H}
$$
from which it follows, by the Kolmogorov continuity criterion, that $\hat m^H_t(u)$ is H\"older continuous of any order $\gamma<H$. Next we show that $\hat m^H_t(u)$ cannot be H\"older continuous of any order $\gamma > H$ at $t=u$. 
Since 
$$
\hat r_H(t,t|u) = \int_u^t \left[k_H(t,v)\right]^2\,\d v,
$$
Proposition \ref{pro:s_asym} gives
$$
\int_u^t \left[k_H(t,v)\right]^2\,\d v = \frac{d^2_H}{2H}|t-u|^{2H} + o\left(|t-u|^{2H}\right).
$$
Now it can be shown that for $H\neq \frac12$ we have
$\frac{d^2_H}{2H} < 1$, and hence we also have
$$
\int_0^u \left[k_H(t,v)-k_H(u,v)\right]^2 \, \d v = \left(1-\frac{d^2_H}{2H}\right)(t-u)^{2H} + o\left(|t-u|^{2H}\right).
$$
In particular, this shows that 
$$
\E\left[\left(\hat m^H_t(u)- \hat m^H_u(u)\right)^2\right] 
= 
\int_0^u \left[k_H(t,v) - k_H(u,v)\right]^2 \,\d v \geq c_H|t-u|^{2H}.
$$
Consequently, the claim follows from the sharpness of the Kolmogorov continuity criterion for Gaussian processes (see \cite{Azmoodeh-Sottinen-Viitasaari-Yazigi-2014}).

Let us then consider the conditional process $\hat B^H(u)$.  Since the conditional mean $\hat m^H(u)$ is H\"older continuous with index $\gamma$ if and only if $\gamma<H$, we may consider the centered conditional process $\bar B^H(u) = \hat B^H(u)-\hat m^H(u)$. 
Since
$$
\E\left[\left(\bar B^H_t(u)-\bar B^H_s(u)\right)^2\right]
=
|t-s|^{2H}-\int_0^u \left[k_H(t,v)-k_H(s,v)\right]^2 \, \d v,
$$
the claim follows with the same arguments as in the conditional mean case.
\end{proof}

%%%%%%%%%%%%%%%%%%%%%%%%%%%%%%%%%%%%%%%%%%%%%%%%%%%%%%%%%%%%%%%%%%%%%%%%%%%%%%%
%%%%%%%%%%%%%%%%%%%%%%%%%%%%%%%%%%%%%%%%%%%%%%%%%%%%%%%%%%%%%%%%%%%%%%%%%%%%%%%
\section{Numerical illustrations}
In this section we provide some numerical illustrations on the behaviour of conditional covariance function $\hat r_H(t,s|u)$. Especially, we illustrate asymptotics claimed in Propositions \ref{pro:0_asym} and \ref{pro:s_asym}. 
For simplicity, we put $t=s=1$. Figure 1 illustrates $\hat r_H(1,1|u)$ for values $u\in[0,1]$ with paremeters $H=0.25$ and $H=0.75$. The convexity for $H=0.75$ and non-convexity of $H=0.25$ is clearly visible.
\begin{figure}[!htbp]
  \centering
  \subfloat[$H=0.75$.]{\includegraphics[width=0.495\textwidth]{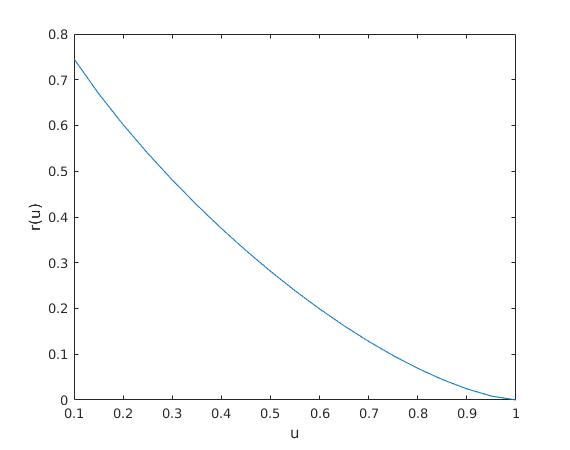}\label{fig:r_big}}
  \hfill
  \subfloat[$H=0.25$.]{\includegraphics[width=0.495\textwidth]{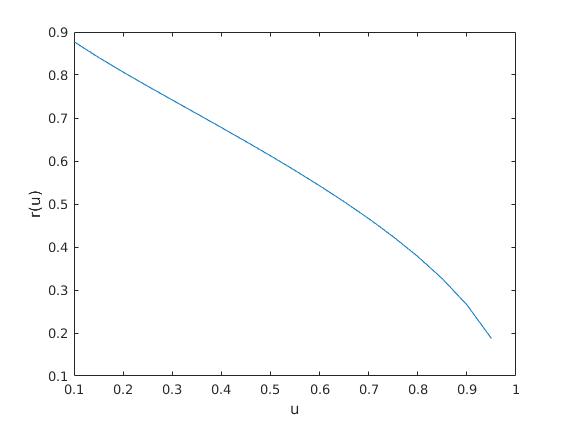}\label{fig:r_small}}
  \caption{$\hat r_H(1,1|u)$ with different values of $H$.}
\end{figure}
Figure 2 illustrates no-information asymptotics for $H=0.75$ and Figure 3 illustrates no-information asymptotics for $H=0.25$. In both pictures we have used short notation 
$$
g(u) = \frac{8-8H}{d_H^2}\frac{1-\hat r_H(1,1|u)}{u^{2H}}.
$$
By Proposition \ref{pro:0_asym}, $g(u)$ should converge towards 1 for $H>\frac12$ and towards constant for $H<\frac12$, which is clearly visible from Figures 2 and 3.
\begin{figure}[!tbp]
  \centering
  \subfloat[$0< u\leq 0.1$.]{\includegraphics[width=0.495\textwidth]{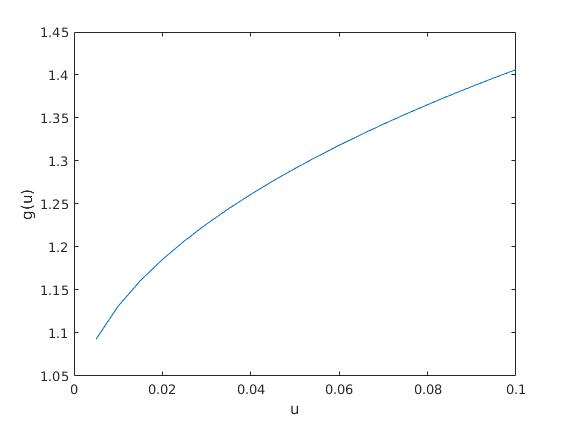}\label{fig:no_info_big}}
  \hfill
  \subfloat[$0.1\leq u\leq 1$.]{\includegraphics[width=0.495\textwidth]{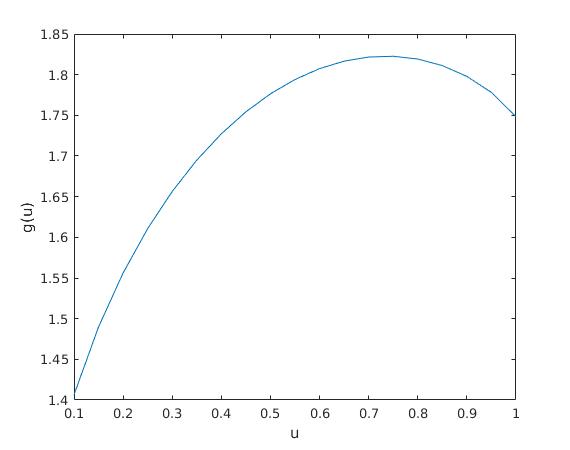}\label{fig:no_info_big_scale}}
  \caption{$g(u)$ against $u$ for $H=0.75$.}
\end{figure}
\begin{figure}[!tbp]
  \centering
  \subfloat[$0< u \leq 0.1$.]{\includegraphics[width=0.495\textwidth]{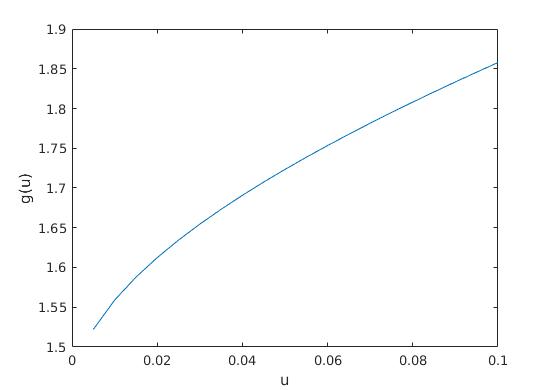}\label{fig:no_info_small}}
  \hfill
  \subfloat[$0.1\leq u \leq 1$.]{\includegraphics[width=0.495\textwidth]{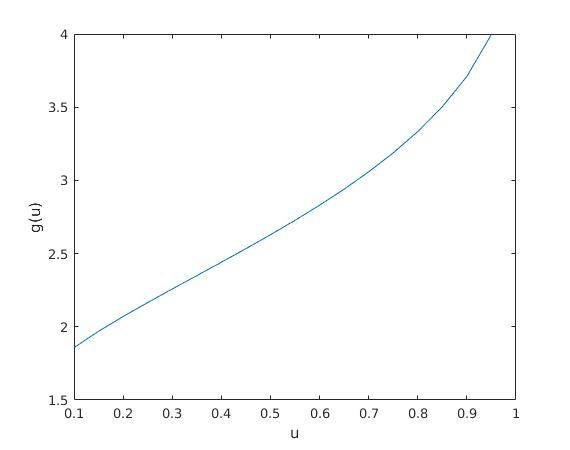}\label{fig:no_info_small_scale}}
  \caption{$g(u)$ against $u$ for $H=0.25$.}
\end{figure}
Finally, the full-information asymptotics given in Proposition \ref{pro:s_asym} is illustrated in Figure 4, where the function 
$$
f(u) = \frac{2H}{d_H^2}\frac{\hat r_H(1,1|u)}{(1-u)^{2H}}
$$
is plotted against $u$.
\begin{figure}[!tbp]
  \centering
  \subfloat[$H=0.75$.]{\includegraphics[width=0.495\textwidth]{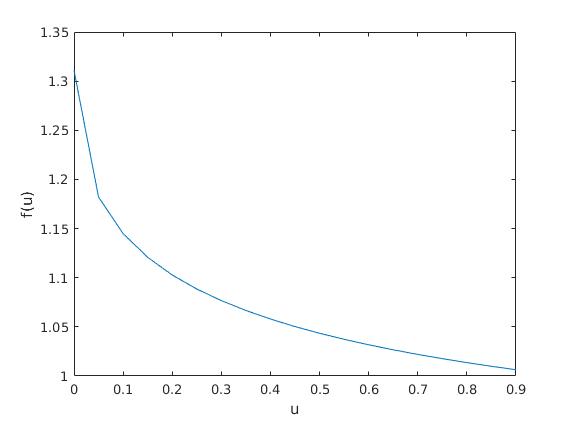}\label{fig:full_info_big}}
  \hfill
  \subfloat[$H=0.25$.]{\includegraphics[width=0.495\textwidth]{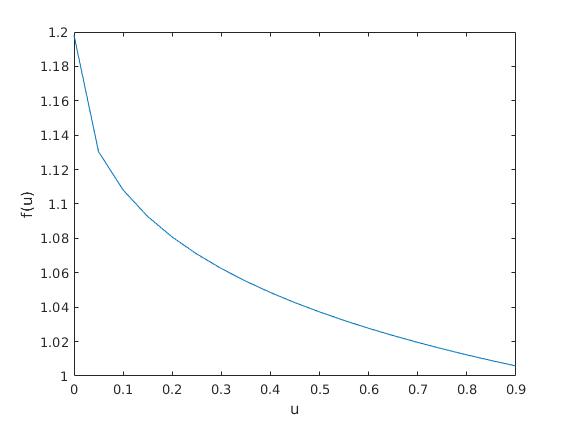}\label{fig:full_info_small}}
  \caption{$f(u)$ against $u$ for $H=0.25$ and $H=0.75$.}
\end{figure}
The convergence towards constant 1 is clearly visible in Figure 4. 

%%%%%%%%%%%%%%%%%%%%%%%%%%%%%%%%%%%%%%%%%%%%%%%%%%%%%%%%%%%%%%%%%%%%%%%%%%%%%%%
%\pagebreak
%%%%%%%%%%%%%%%%%%%%%%%%%%%%%%%%%%%%%%%%%%%%%%%%%%%%%%%%%%%%%%%%%%%%%%%%%%%%%%%

\section{Proofs of Propositions \ref{pro:func_decreasing}, \ref{pro:0_asym} and \ref{pro:s_asym}}
\label{sec:proofs}

\begin{proof}[Proof of Proposition \ref{pro:func_decreasing}]
Let $u<s\le t$. From \eqref{eq:cond_cov} we observe that 
$$
\frac{\partial}{\partial u} \hat r_H(t,s|u) = - k_H(t,u)k_H(s,u).
$$
Since $k_H(t,u)$ is infinitely differentiable with respect to $u$ for all $t$, it follows that $\hat r_H(t,s|u)$ is infinitely differentiable in $u$. 
Furthermore, since $k_H(t,u)>0$ for all $t$ and $u$, we observe that $\frac{\partial}{\partial u} \hat r_H(t,s|u) < 0$. Hence $\hat r_H(t,s|u)$ is strictly decreasing in $u$. 

Next we prove the convexity in $u$ of $\hat r_H(t,s|u)$ for $H>\frac{1}{2}$. For this it is sufficient to show that $-k_H(t,u)k_H(s,u)$ is increasing in $u$. Hence, it suffices to show that $k_H(t,s)$ is decreasing in $s$. Indeed, then $-k_H(t,u)k_H(s,u)$ is increasing in $u$ since $k_H(t,s)\geq 0$. By \cite[Eq. (1.2)]{Jost-2006} we have
$$
k_H(t,s) = C_H (t-s)^{H-\frac{1}{2}} F\left(\frac{1}{2} - H, H-\frac{1}{2}, H+\frac{1}{2}; \frac{s-t}{s}\right),
$$
where $F$ denotes the Gauss hypergeometric function. Denote $v = 1- \frac{s}{t}$ and let $t$ be fixed. Then 
$$
k_H(t,s) = k_H(t, t(1-v)) = t^{H-\frac{1}{2}}v^{H-\frac{1}{2}}F\left(\frac{1}{2} - H, H-\frac{1}{2}, H+\frac{1}{2}; \frac{v}{v-1}\right),
$$ 
where $v\in[0,1]$. By \cite[p. 269, eq. (8.2.9)]{Beals-Wong-2010} and the symmetry of Gauss hypergeometric function with respect to first two parameters, we have
$$
F\left(\frac{1}{2} - H, H-\frac{1}{2}, H+\frac{1}{2}; \frac{v}{v-1}\right) = (1-v)^{\frac{1}{2} -H}F\left(1,\frac{1}{2} - H, H+\frac{1}{2}; v\right),
$$
and hence it suffices to show that 
$$
\left(\frac{v}{1-v}\right)^{H-\frac{1}{2}}F\left(1,\frac{1}{2}-H,H+\frac{1}{2};v\right)
$$
is increasing as a function of $v$. To show this, we use the Euler integral formula \cite[p. 271, Proposition 8.3.1.]{Beals-Wong-2010}
$$
F(a,b,c;v) = \frac{1}{\B(b,c-b)}\int_0^1 x^{b-1}(1-x)^{c-b-1}(1-vx)^{-a}\d x
$$
provided that $|v|<1$ and $c>b>0$, where $\B(a,b)$ denotes the Beta function. Hence we have
\begin{eqnarray*}
\lefteqn{\left(\frac{v}{1-v}\right)^{H-\frac{1}{2}}F\left(1,\frac{1}{2}-H,H+\frac{1}{2};v\right)}  \\
&=& \frac{1}{\B\left(1,H-\frac{1}{2}\right)}\left(\frac{v}{1-v}\right)^{H-\frac{1}{2}}\int_0^1 (1-x)^{H-\frac{3}{2}}(1-vx)^{H-\frac{1}{2}}\d x\\
&=&
\frac{1}{\B\left(1,H-\frac{1}{2}\right)}\int_0^1 (1-x)^{H-\frac{3}{2}} \left(\frac{v}{1-v}(1-vx)\right)^{H-\frac{1}{2}}\d x.
\end{eqnarray*}
Now it is straightforward to see that, for any $x\in(0,1)$,
$$
\frac{v}{1-v} - \frac{v^2}{1-v}x
$$
is an increasing function in $v$. Consequently, $k_H(t,s) = k_H(t, t(1-v))$ is also increasing as a function of $v$, and thus $-k_H(t,u)k_H(s,u)$ is increasing in $u$, which shows that, for fixed $t$ and $s$, $\hat r_H(t,s|u)$ is a convex function.
\end{proof}

%%%%%%%%%%%%%%%%%%%%%%%%%%%%%%%%%%%%%%%%%%%%%%%%%%%%%%%%%%%%%%%%%%%%%%%%%%%%%%%

\begin{proof}[Proof of Proposition \ref{pro:0_asym}]
Denote
$$%\begin{equation}\label{eq:beta}
\beta_H(\tau)
=
\int_1^{\tau} w^{H-\frac{3}{2}}(w-1)^{H-\frac{1}{2}}\, \d w.
$$%\end{equation}
Then, by using the change of variable $w = \frac{z}{v}$ in \eqref{eq:kernel} we can write
$$
k_H(t,v) = d_H\left[\left(\frac{t}{v}\right)^{H-\frac{1}{2}}(t-v)^{H-\frac{1}{2}}
-\left(H-\frac{1}{2}\right)v^{H-\frac{1}{2}}\beta_H\left(\frac{t}{v}\right)\right].  
$$
Then, from \eqref{eq:cond_cov} it follows that
\begin{eqnarray*}
\lefteqn{\hat r_H(t,s|u) - r_H(t,s)}\\
&=& -d_H^2\int_0^u \left[I^H_1(t,s,v) + I^H_2(t,s,v) + I^H_3(t,s,v) + I^H_4(t,s,v)\right] \d v,
\end{eqnarray*}
where
\begin{eqnarray*}
I^H_1(t,s,v) &=& \left(\frac{t}{v}\right)^{H-\frac{1}{2}}\left(\frac{s}{v}\right)^{H-\frac{1}{2}}(t-v)^{H-\frac{1}{2}}(s-v)^{H-\frac{1}{2}}, \\
I^H_2(t,s,v) &=& \left(\frac{1}{2}-H\right)t^{H-\frac{1}{2}}(t-v)^{H-\frac{1}{2}}\beta_H\left(\frac{s}{v}\right), \\
I^H_3(t,s,v) &=& \left(\frac{1}{2}-H\right)s^{H-\frac{1}{2}}(s-v)^{H-\frac{1}{2}} \beta_H\left(\frac{t}{v}\right), \\
I^H_4(t,s,v) &=& \left(H-\frac{1}{2}\right)^2 v^{2H-1} \beta_H\left(\frac{s}{v}\right)\beta_H\left(\frac{t}{v}\right).
\end{eqnarray*}

Consider first the term $I^H_1(t,s,v)$.
Recall that $|a^\gamma - b^\gamma| \le |a-b|^\gamma$ for any $\gamma\in(0,1)$. Consequently, for $H>\frac{1}{2}$,
\begin{equation}
\label{eq:basic}
|(t-v)^{H-\frac{1}{2}} - t^{H-\frac{1}{2}}| \le v^{H-\frac{1}{2}},
\end{equation}
which implies that
\begin{equation}
\label{eq:basic_large_H}
(t-v)^{H-\frac{1}{2}} = t^{H-\frac{1}{2}} + O\left(v^{H-\frac{1}{2}}\right).
\end{equation}
Similarly, for $H<\frac{1}{2}$, 
\begin{eqnarray*}
|(t-v)^{H-\frac{1}{2}} - t^{H-\frac{1}{2}}| &=& \left|\frac{1}{(t-v)^{\frac{1}{2} -H}} - \frac{1}{t^{\frac{1}{2}-H}}\right|\\
&\le& \left|\frac{1}{t-v} - \frac{1}{t}\right|^{\frac{1}{2} -H}\\
&=& \left(\frac{v}{(t-v)t}\right)^{\frac{1}{2} -H}.
\end{eqnarray*}
Consequently, as $v\rightarrow 0$, we have
\begin{equation}
\label{eq:basic_small_H}
(t-v)^{H-\frac{1}{2}} = t^{H-\frac{1}{2}} + O\left(v^{\frac{1}{2}-H}\right).
\end{equation}
Hence, as $v \to 0$,
$$
I^H_1(t,s,v) = (ts)^{2H-1}v^{1-2H} + o(v^{1-2H})
$$
and, consequently,
\begin{equation}
\label{eq:I1_0}
\int_0^u I^H_1(t,s,v)\,\d v = \frac{(ts)^{2H-1}}{2-2H}u^{2-2H} + o(u^{2-2H}).
\end{equation}

Consider then the next three remaining terms $I^H_2(t,s,v)$, $I^H_3(t,s,v)$ and $I^H_4(t,s,v)$. 
We begin with the case $H<\frac{1}{2}$. Now
$\beta_H(\infty)<\infty$ and
\begin{eqnarray*}
\left|\beta_H\left(\frac{t}{v}\right)- \beta_H(\infty)\right| 
&=& \int_{\frac{t}{v}}^\infty w^{H-\frac{3}{2}}(w-1)^{H-\frac{1}{2}}\, \d w\\
&\le& \int_{\frac{t}{v}}^\infty (w-1)^{2H-2}\, \d w\\
&=& C_H \left(\frac{t}{v} - 1\right)^{2H-1}\\
&\le& C_{H,t} v^{1-2H}
\end{eqnarray*}
for small enough $v$. Consequently, 
\begin{equation}
\label{eq:int_asym_small_H}
\beta_H\left(\frac{t}{v}\right) = \beta_H(\infty) + O\left(v^{1-2H}\right).
\end{equation}
By using this together with \eqref{eq:basic_small_H} we get
$$
I^H_2(t,s,v) = \left(\frac{1}{2}-H\right)t^{H-\frac{1}{2}}\beta_H\left(\infty\right) + O(v^{\frac{1}{2}-H})
$$
from which it follows that 
$$
\int_0^u I^H_2(t,s,v) \,\d v = O(u) = o(u^{2H}).
$$
Moreover, with the same arguments we observe
$$
\int_0^u I^H_3(t,s,v)\d v = o(u^{2H})
$$
and by \eqref{eq:I1_0} we also have
$$
\int_0^u I^H_1(t,s,v)\d v = O(u^{2-2H}) = o\left(u^{2H}\right).
$$
Finally, for $I^H_4$ we have, again thanks to \eqref{eq:int_asym_small_H},
$$
I^H_4(t,s,v) = \left(H-\frac{1}{2}\right)^2 \left(\int_1^\infty w^{H-\frac{3}{2}}(w-1)^{H-\frac{1}{2}}\d w\right)^2 v^{2H-1} + O(1)
$$
from which the claim follows by integrating with respect to $v$ over the interval $[0,u]$ for $H<\frac{1}{2}$.
Let then $H>\frac{1}{2}$. We have
\begin{eqnarray}
\nonumber
\beta_H\left(\frac{t}{v}\right) 
&=& \int_1^{\frac{t}{v}} w^{2H-2}\,\d w
+ \int_1^{\frac{t}{v}} w^{H-\frac{3}{2}}\left[(w-1)^{H-\frac{1}{2}} - w^{H-\frac{1}{2}}\right]\,\d w \nonumber\\
\label{eq:I2_0_start}
&=& \frac{t^{2H-1}}{2H-1}v^{1-2H} + O\left(v^{\frac{1}{2}-H}\right),
\end{eqnarray}
where the last equality follows from \eqref{eq:basic}. By using \eqref{eq:basic_large_H} again we hence observe that 
$$
I^H_2(t,s,v) = \left(\frac{1}{2}-H\right)\frac{(ts)^{2H-1}}{2H-1}v^{1-2H} + O(v^{\frac{1}{2}-H}).
$$
Since $O\left(u^{\frac{3}{2}-H}\right) = o\left(u^{2-2H}\right)$ for $H>\frac{1}{2}$, we have
\begin{equation}
\label{eq:I2_0}
\int_0^u I^H_2(t,s,v)\,\d v = \left(\frac{1}{2}-H\right)\frac{(ts)^{2H-1}}{(2H-1)(2-2H)}u^{2-2H} + o(u^{2-2H}).
\end{equation}
Similarly, we observe
\begin{equation}
\label{eq:I3_0}
\int_0^u I^H_3(t,s,v)\,\d v 
= \left(\frac{1}{2}-H\right)\frac{(ts)^{2H-1}}{(2H-1)(2-2H)}u^{2-2H} + o(u^{2-2H}).
\end{equation}
For $I^H_4(t,s,v)$, we obtain by \eqref{eq:I2_0_start} that
$$
I^H_4(t,s,v) = \left(H-\frac{1}{2}\right)^2\frac{(ts)^{2H-1}}{(2H-1)^2}v^{1-2H} + O(v^{\frac{1}{2} -H})
$$
and hence
\begin{equation}
\label{eq:I4_0}
\int_0^u I^H_4(t,s,v)\,\d v 
= \left(H-\frac{1}{2}\right)^2\frac{(ts)^{2H-1}}{(2-2H)(2H-1)^2}u^{2-2H} + o(u^{2-2H}).
\end{equation}
Now the result follows by combining equations \eqref{eq:I2_0}--\eqref{eq:I4_0} with \eqref{eq:I1_0} together with some simplifications. 
\end{proof}

%%%%%%%%%%%%%%%%%%%%%%%%%%%%%%%%%%%%%%%%%%%%%%%%%%%%%%%%%%%%%%%%%%%%%%%%%%%%%%%

\begin{proof}[Proof of Proposition \ref{pro:s_asym}]
Let $\beta_H$ and $I^H_i(t,s,v)$, $i=1,2,3,4$, be like in the proof of Proposition \ref{pro:0_asym}.

We begin by showing that the terms $I^H_2(t,s,v)$ and $I^H_4(t,s,v)$ are negligible. For this note that 
\begin{equation}
\beta_H\left(\frac{s}{v}\right)
\le C_H\left(\frac{s}{v}-1\right)^{H+\frac{1}{2}}
\le C_H(s-v)^{H+\frac{1}{2}} \label{eq:integral_estimate1}
\end{equation}
for $v$ close enough to $s$. Consequently, for $t>s$ we have
$$
I^H_2(t,s,v) = O\left((s-v)^{H+\frac{1}{2}}\right)
$$
from which it follows that
$$
\int_u^s I^H_2(t,s,v)\,\d v = o\left((s-v)^{H+\frac{1}{2}}\right).
$$
Similarly, for $t=s$ we have
$$
I^H_2(s,s,v) = O\left((s-v)^{2H}\right)
$$
and thus
$$
\int_u^s I^H_2(s,s,v)\,\d v = o\left((s-v)^{2H}\right).
$$
This implies that the term $I^H_2(t,s,v)$ is negligible. For terms $I^H_3(t,s,v)$ and $I^H_4(t,s,v)$, we first observe that
\begin{equation}
\label{eq:decomposition}
\beta_H\left(\frac{t}{v}\right)
= \beta_H\left(\frac{t}{s}\right)
+\int_{\frac{t}{s}}^{\frac{t}{v}} w^{H-\frac{3}{2}}(w-1)^{H-\frac{1}{2}}\,\d w. 
\end{equation}
Here the first term, denoted by $\beta_H\left(\frac{t}{s}\right)$, is just a constant independent of $v$ and $u$. For the second term we have
\begin{equation}
\int_{\frac{t}{s}}^{\frac{t}{v}} w^{H-\frac{3}{2}}(w-1)^{H-\frac{1}{2}}\,\d w 
\le C_{H,t,s} \left(\frac{t}{v}-\frac{t}{s}\right)^{H+\frac{1}{2}}
\le C_{H,t,s}(s-v)^{H+\frac{1}{2}}\label{eq:integral_estimate2} 
\end{equation}
for $v$ close to $s$. Hence the term $I^H_4(t,s,v)$ is also negligible. Indeed, combining estimates \eqref{eq:integral_estimate1} and \eqref{eq:integral_estimate2} we get
$$
\int_u^s I^H_4(t,s,v)\,\d v \le C_{H,t,s} \int_u^s \left[\beta_H\left(\frac{t}{s}\right)(s-v)^{H+\frac{1}{2}} + (s-v)^{2H+1}\right]\, \d v.
$$
Consequently, for $t>s$ we have
$$
\int_u^s I^H_4(t,s,v)\,\d v = o\left((s-u)^{H+\frac{1}{2}}\right)
$$
and for $t=s$, thanks to the fact $\beta_H(1)=0$, we have
$$
\int_u^s I^H_4(s,s,v)\d v = o\left((s-u)^{2H}\right).
$$

Let us next study the term $I^H_3(t,s,v)$. By using the decomposition \eqref{eq:decomposition} and the estimate \eqref{eq:integral_estimate2} we obtain that 
$$
I^H_3(t,s,v) = \left(\frac{1}{2}-H\right)s^{H-\frac{1}{2}}(s-v)^{H-\frac{1}{2}}\beta_H\left(\frac{t}{s}\right) + O((s-v)^{2H}).
$$
Hence for $t>s$ we have 
$$
\int_u^s I^H_3(t,s,v)\,\d v = \frac{\frac{1}{2}-H}{\frac{1}{2} + H}\beta_H\left(\frac{t}{s}\right)s^{H-\frac{1}{2}}(s-u)^{H+\frac{1}{2}} + o\left((s-u)^{H+\frac{1}{2}}\right)
$$
and for $t=s$ we have
$$
\int_u^s I^H_3(s,s,v)\,\d v = o\left((s-u)^{2H}\right).
$$

To conclude the proof, it remains to study the term $I^H_1(t,s,v)$. We write
$$
I^H_1(t,s,v) = \left(\frac{t}{s}\right)^{H-\frac{1}{2}}\left(\frac{s}{v}\right)^{2H-1}(t-v)^{H-\frac{1}{2}}(s-v)^{H-\frac{1}{2}}.
$$
Furthermore, using similar analysis as above we observe that, for $v$ close to $s$, we have
$$
\left(\frac{s}{v}\right)^{2H-1} - 1 \le C_H(s-v)^{2H-1}
$$
for $H>\frac{1}{2}$ and 
$$
\left(\frac{s}{v}\right)^{2H-1} - 1 \le C_H(s-v)^{1-2H}
$$
for $H<\frac{1}{2}$. Thus instead of $I^H_1(t,s,v)$ it suffices to consider 
$$
\left(\frac{t}{s}\right)^{H-\frac{1}{2}}(t-v)^{H-\frac{1}{2}}(s-v)^{H-\frac{1}{2}},
$$
from which we easily observe that, for $t=s$, we have
$$
\int_u^s I^H_1(t,s,v)\,\d v = \frac{1}{2H}\left(\frac{t}{s}\right)^{H-\frac{1}{2}}(s-u)^{2H} + o\left((s-u)^{2H}\right).
$$
For $t>s$ we write
\begin{eqnarray*}
\lefteqn{\int_u^s (t-v)^{H-\frac{1}{2}}(s-v)^{H-\frac{1}{2}}\,\d v} \\
&=& \int_u^s (t-s)^{H-\frac{1}{2}}(s-v)^{H-\frac{1}{2}}\d v\\
& &+\int_u^s \left[(t-v)^{H-\frac{1}{2}}-(t-s)^{H-\frac{1}{2}}\right](s-v)^{H-\frac{1}{2}}\d v\\
&=&\frac{1}{H+\frac{1}{2}}(t-u)^{H-\frac{1}{2}}(s-u)^{H+\frac{1}{2}}\\
& &+\int_u^s \left[(t-v)^{H-\frac{1}{2}}-(t-s)^{H-\frac{1}{2}}\right](s-v)^{H-\frac{1}{2}}\d v.
\end{eqnarray*}
For $H>\frac{1}{2}$ we have
\begin{equation*}
\left|(t-v)^{H-\frac{1}{2}}-(t-s)^{H-\frac{1}{2}}\right| \leq (s-v)^{H-\frac{1}{2}},
\end{equation*}
from which it follows that
$$
\int_u^s \left[(t-v)^{H-\frac{1}{2}}-(t-s)^{H-\frac{1}{2}}\right](s-v)^{H-\frac{1}{2}}\,\d v = O\left((s-u)^{2H}\right) = o\left((s-u)^{H+\frac{1}{2}}\right).
$$
Similarly, for $H<\frac{1}{2}$ we have
\begin{equation*}
\left|(t-v)^{H-\frac{1}{2}}-(t-s)^{H-\frac{1}{2}}\right| 
\le \left(\frac{v-s}{(t-v)(t-s)}\right)^{H-\frac{1}{2}}
\le (t-s)^{2H-1}(s-v)^{\frac{1}{2} -H}.
\end{equation*}
Hence 
$$
\int_u^s \left[(t-v)^{H-\frac{1}{2}}-(t-u)^{H-\frac{1}{2}}\right](s-v)^{H-\frac{1}{2}}\,\d v 
=
O\left(s-u\right)
= 
o\left((s-u)^{H+\frac{1}{2}}\right).
$$
Combining the above estimates we thus observed that, in the case $t>s$, 
$$
\int_u^s I^H_1(t,s,v)\,\d v = \frac{1}{H+\frac{1}{2}}\left(\frac{t}{s}\right)^{H-\frac{1}{2}}(t-s)^{H-\frac{1}{2}}(s-u)^{H+\frac{1}{2}} + o\left((s-u)^{H+\frac{1}{2}}\right).
$$
\end{proof}
%%%%%%%%%%%%%%%%%%%%%%%%%%%%%%%%%%%%%%%%%%%%%%%%%%%%%%%%%%%%%%%%%%%%%%%%%%%%%%%
%\section{Discussion}

%%%%%%%%%%%%%%%%%%%%%%%%%%%%%%%%%%%%%%%%%%%%%%%%%%%%%%%%%%%%%%%%%%%%%%%%%%%%%%%
%%%%%%%%%%%%%%%%%%%%%%%%%%%%%%%%%%%%%%%%%%%%%%%%%%%%%%%%%%%%%%%%%%%%%%%%%%%%%%%
\bibliographystyle{siam}
\bibliography{pipliateekki}

\begin{thebibliography}{10}

\bibitem{Anh-Inoue-2004}
{\sc V.~V. Anh and A.~Inoue}, {\em Prediction of fractional {B}rownian motion
  with {H}urst index less than {$1/2$}}, Bull. Austral. Math. Soc., 70 (2004),
  pp.~321--328.

\bibitem{Azmoodeh-Sottinen-Viitasaari-Yazigi-2014}
{\sc E.~Azmoodeh, T.~Sottinen, L.~Viitasaari, and A.~Yazigi}, {\em Necessary
  and sufficient conditions for {H}\"older continuity of {G}aussian processes},
  Statist. Probab. Lett., 94 (2014), pp.~230--235.

\bibitem{Beals-Wong-2010}
{\sc R.~Beals and R.~Wong}, {\em Special functions}, vol.~126 of Cambridge
  Studies in Advanced Mathematics, Cambridge University Press, Cambridge, 2010.
\newblock A graduate text.

\bibitem{Biagini-Hu-Oksendal-Zhang-2008}
{\sc F.~Biagini, Y.~Hu, B.~{\O}ksendal, and T.~Zhang}, {\em Stochastic calculus
  for fractional {B}rownian motion and applications}, Probability and its
  Applications (New York), Springer-Verlag London, Ltd., London, 2008.

\bibitem{Bogachev-1998}
{\sc V.~I. Bogachev}, {\em Gaussian measures}, vol.~62 of Mathematical Surveys
  and Monographs, American Mathematical Society, Providence, RI, 1998.

\bibitem{Duncan-2006}
{\sc T.~E. Duncan}, {\em Prediction for some processes related to a fractional
  {B}rownian motion}, Statist. Probab. Lett., 76 (2006), pp.~128--134.

\bibitem{Fink-Kluppelberg-Zahle-2013}
{\sc H.~Fink, C.~Kl\"uppelberg, and M.~Z\"ahle}, {\em Conditional distributions
  of processes related to fractional {B}rownian motion}, J. Appl. Probab., 50
  (2013), pp.~166--183.

\bibitem{Gripenberg-Norros-1996}
{\sc G.~Gripenberg and I.~Norros}, {\em On the prediction of fractional
  {B}rownian motion}, J. Appl. Probab., 33 (1996), pp.~400--410.

\bibitem{Janson-1997}
{\sc S.~Janson}, {\em Gaussian {H}ilbert spaces}, vol.~129 of Cambridge Tracts
  in Mathematics, Cambridge University Press, Cambridge, 1997.

\bibitem{Jost-2006}
{\sc C.~Jost}, {\em Transformation formulas for fractional {B}rownian motion},
  Stochastic Process. Appl., 116 (2006), pp.~1341--1357.

\bibitem{LaGatta-2013}
{\sc T.~LaGatta}, {\em Continuous disintegrations of {G}aussian processes},
  Theory Probab. Appl., 57 (2013), pp.~151--162.

\bibitem{Mishura-2008}
{\sc Y.~S. Mishura}, {\em Stochastic calculus for fractional {B}rownian motion
  and related processes}, vol.~1929 of Lecture Notes in Mathematics,
  Springer-Verlag, Berlin, 2008.

\bibitem{Molchan-2003}
{\sc G.~M. {Molchan}}, {\em {Historical comments related to fractional Brownian
  motion.}}, in {Theory and applications of long-range dependence}, Boston, MA:
  Birkh\"auser, 2003, pp.~39--42.

\bibitem{Norros-Valkeila-Virtamo-1999}
{\sc I.~Norros, E.~Valkeila, and J.~Virtamo}, {\em An elementary approach to a
  {G}irsanov formula and other analytical results on fractional {B}rownian
  motions}, Bernoulli, 5 (1999), pp.~571--587.

\bibitem{Pipiras-Taqqu-2001}
{\sc V.~Pipiras and M.~S. Taqqu}, {\em Are classes of deterministic integrands
  for fractional {B}rownian motion on an interval complete?}, Bernoulli, 7
  (2001), pp.~873--897.

\bibitem{Samko-Kilbas-Marichev-1993}
{\sc S.~G. Samko, A.~A. Kilbas, and O.~I. Marichev}, {\em Fractional integrals
  and derivatives}, Gordon and Breach Science Publishers, Yverdon, 1993.
\newblock Theory and applications, Edited and with a foreword by S. M.
  Nikol'ski{\u\i}, Translated from the 1987 Russian original, Revised by the
  authors.

\bibitem{Sottinen-Yazigi-2014}
{\sc T.~Sottinen and A.~Yazigi}, {\em Generalized {G}aussian bridges},
  Stochastic Process. Appl., 124 (2014), pp.~3084--3105.

\end{thebibliography}
\end{document}